\documentclass[12pt,reqno]{amsart}
\usepackage{amsmath, enumerate, amsthm, amssymb,verbatim,fullpage,multicol,cases}
\usepackage{graphicx}
\usepackage{hyperref}
\usepackage[table]{xcolor}
\usepackage[T1]{fontenc}
\usepackage{appendix}
\usepackage[utf8]{inputenc}
\usepackage{soul}
\usepackage{braket}
\usepackage{soul}
\usepackage{enumitem}
\usepackage{multirow}
\usepackage{adjustbox}
\usepackage{array}

\usepackage{txfonts}		
\usepackage[noend]{algpseudocode}
\usepackage{amsfonts}
\usepackage{hyperref}
\usepackage{float}
\usepackage{amssymb}
\usepackage{amsmath}
\usepackage{amsthm}
\usepackage[linesnumbered,ruled]{algorithm2e}

\usepackage{amscd}
\usepackage{verbatim}
\usepackage{tkz-graph}
\usetikzlibrary{calc}

\usepackage[small,bf]{caption}
\usepackage{graphicx}
\usepackage{pdflscape}
\usepackage{enumerate}
\makeatletter
\renewcommand*\env@matrix[1][c]{\hskip -\arraycolsep
  \let\@ifnextchar\new@ifnextchar
  \array{*\c@MaxMatrixCols #1}}
\makeatother

\newtheorem{theorem}{Theorem}[section] 
\newtheorem{lemma}[theorem]{Lemma}
\newtheorem{proposition}[theorem]{Proposition}
\newtheorem{corollary}[theorem]{Corollary}

\theoremstyle{definition}
\newtheorem{example}[theorem]{Example}

\theoremstyle{definition}
\newtheorem{definition}[theorem]{Definition}

\theoremstyle{definition}

\theoremstyle{plain}

\theoremstyle{remark}

\numberwithin{subcase}{case}

\def\m(h){{m(h)}} 

\def\PSG{{P(S;G)}}
\def\PSn{{P(S;n)}}

\newcommand{\lab}{\mathfrak{L}}

\numberwithin{figure}{section}
\numberwithin{theorem}{section}
\newcommand{\tcr}[1]{\textcolor{red}{#1}}

\newcommand{\ol}[1]{\overline{#1}}
\makeatletter
\def\BState{\State\hskip-\ALG@thistlm}
\makeatother

\renewenvironment{abstract}
 {\small
 \begin{center}
 \bfseries \abstractname\vspace{-.5em}\vspace{0pt}
 \end{center}
 \list{}{
 \setlength{\leftmargin}{2cm}%
 \setlength{\rightmargin}{\leftmargin}%
 }%
 \item\relax}
 {\endlist}
 
\title{Counting peaks on graphs}
 
\author{Alexander Diaz-Lopez}
\address{Department of Mathematics \& Statistics\\
  Villanova University\\
  Villanova, PA 19085}
\email[A.~Diaz-Lopez]{alexander.diaz-lopez@villanova.edu}

\author{Lucas Everham}
\address{Department of Mathematics\\ 
Florida Gulf Coast University\\ 
Fort Myers, Florida 33965}
\email[L.Everham]{lfeverham4783@eagle.fgcu.edu}

\author{Pamela E. Harris}
\address{Department of Mathematics and Statistics\\ 
Williams College\\ 
Williamstown, Massachusetts 01267}
\email[P. E. Harris]{pamela.e.harris@williams.edu}

\author{Erik Insko}
\address{Department of Mathematics\\ 
Florida Gulf Coast University\\ 
Fort Myers, Florida 33965}
\email[E. Insko]{einsko@fgcu.edu}

\author{Vincent Marcantonio}
\address{Department of Mathematics\\ 
Florida Gulf Coast University\\ 
Fort Myers, Florida 33965}
\email[V. Marcantonio]{vrmarcantonio7740@eagle.fgcu.edu}

\author{Mohamed Omar}
\address{Department of Mathematics\\ 
Harvey Mudd College\\ 
Claremont, California 91711}
\email[M. Omar]{omar@g.hmc.edu}

\thanks{The second, fourth, and fifth authors were supported in part by a Seidler Student/Faculty Undergraduate Scholarly Collaboration Fellowship Program at Florida Gulf Coast University. The third named author was partially supported by NSF grant DMS1620202.}
\begin{document}

\maketitle
\begin{abstract}
Given a graph $G$ with $n$ vertices and a bijective labeling of the vertices using the integers $1,2,\ldots, n$, we say $G$ has a peak at vertex $v$ if the degree of $v$ is greater than or equal to 2, and if the label on $v$ is larger than the label of all its neighbors. Fix an enumeration of the vertices of $G$ as $v_1,v_2,\ldots, v_{n}$ and a fix a set $S\subset V(G)$. We want to determine the number of distinct bijective labelings of the vertices of $G$, such that the vertices in $S$ are precisely the peaks of $G$. The set $S$ is called the \emph{peak set of the graph} $G$, and the set of all labelings with peak set $S$ is denoted by $\PSG$. This definition generalizes the study of peak sets of permutations, as that work is the special case of $G$ being the path graph on $n$ vertices. In this paper, we present an algorithm for constructing all of the bijective labelings in $\PSG$ for any $S\subseteq V(G)$.
We also explore peak sets in certain families of graphs, including cycle graphs and joins of graphs. 
\end{abstract}


\section{Introduction} \label{sec:intro}

Let $[n]:=\{1,2,\ldots, n\}$  and $\mathfrak S_n$ denote the symmetric group on $n$ letters.
We let $\pi=\pi_1\pi_2\cdots \pi_n$ denote the one-line notation for a permutation $\pi\in\mathfrak S_n$ and we say that $\pi$ 
has a \emph{peak} at index $i$ if 
$\pi_{i-1}<\pi_i>\pi_{i+1}$.  The peak set of a permutation $\pi$ is defined as the set \[ P (\pi) = \{i \in [n]\, \vert 
\, \mbox{ $\pi$ has a peak at index $i$}\}. \] 
Given a subset $S \subseteq [n]$
we denote the set of all permutations with peak set $S$ by
\[\PSn = \{ \pi \in \mathfrak S_n \, \vert \,  P(\pi) = S\}.\]  
Peak sets of permutations have been the focus of much research; 
in particular, these sets are useful in studying peak algebras of symmetric groups
\cite{ABN04,ANO06,BH06,N03,P07,S78}, and more recently, enumerating the sets $P(S;n)$ for various $S$ has drawn considerable attention \cite{BBPS14,BBS13,BFT16,CV14,DHIO16,DHIP15,K14}.
In their celebrated paper, Billey, Burdzy, and Sagan \cite{BBS13} developed a recursive formula (whose terms alternate in sign) 
for $|P(S;n)|$ and showed that
\begin{align}
|\PSn| &= 2^{n-|S|-1} p_S(n)\label{eq:peakpoly}  
\end{align}
where $p_S(x)$ is a polynomial of degree 
$ \max(S)-1$ referred to as the \emph{peak polynomial} associated to the set $S$. 
\par

The results in \cite{BBS13} motivated the work of Billey, Burdzy, Pal, and Sagan 
on a probabilistic mass redistribution model on graphs that made thorough use of peak sets of permutations \cite{BBPS14}. Later, the research groups of Castro-Velez, Diaz-Lopez, Orellana, Pastrana, and Zevallos \cite{CV14} and Diaz-Lopez, Harris, Insko, and
Perez-Lavin \cite{ DHIP15} studied peak sets in Coxeter groups of classical types. 
More recently, Diaz-Lopez, Harris, Insko, and Omar developed a recursive formula for $|\PSn|$ that allowed the authors to resolve a conjecture of Billey, Burdzy and Sagan 
claiming that peak polynomials have nonnegative coefficients when expanded in a particular binomial basis \cite{DHIO16};
this newer recursive formula is based on an analysis of the possible positions in which the largest number in the permutation can appear.

A generalization of the results of peaks on permutations, and the focus of our study, is through
the following graph-theoretic lens.
Let $G$ be a graph with $n$ vertices $v_1, \ldots, v_{n}$.  A permutation $\pi=\pi_1\cdots \pi_n \in \mathfrak{S}_n$ corresponds to a 
bijective labeling $\ell_{\pi}:V(G) \to [n]$ by setting $\pi_i$ to be the label of vertex {$v_i$ i.e., $\ell_{\pi}(v_i)=\pi_i$}.  Through this correspondence, we interchangeably refer to a labeling and its corresponding permutation.  We say that a permutation $\pi$ has a peak at the vertex $v_i$ of $G$ if $\ell_{\pi}(v_i)>\ell_{\pi}(v_j)$ for all vertices $v_j$ adjacent to $v_i$, and remark that we do not allow peaks at vertices of degree 1 or 0 so that peak sets on paths agree with the existing literature with peak sets on permutations.

The \emph{$G$-peak set} of a permutation $\pi$ is defined to be the set \[ P_G(\pi) = \{i \in [n]\, \vert 
\, \mbox{ $\pi$ has a peak at the vertex $v_i$}\}. \]
Given $S \subseteq V(G):=\{v_1,\ldots,v_n\}$, we denote the set of all permutations with $G$-peak set $S$~by
\[\PSG = \{ \pi \in \mathfrak S_n \, \vert \,  P_G(\pi) = S\}\] and say $S$ is a $G$-admissible set if $\PSG$ is nonempty. 
 
\begin{example}
Below is a graph $G$ with vertices $v_1,v_2,v_3$ and $v_4$ and four different labelings of $G$. The first two labelings have peak set $S =\{v_1\}$, whereas the last
two have empty peak set.\\
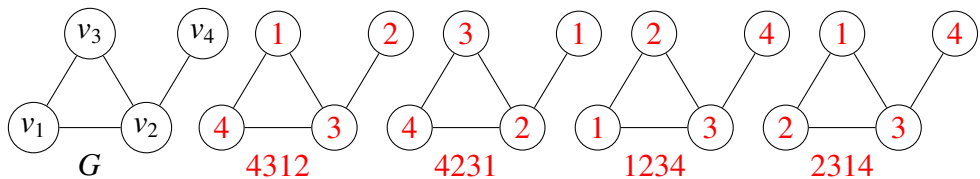
\begin{center}
\begin{tikzpicture}[scale=.50]
    \tikzstyle{ghost node}=[draw=none]
    \tikzset{blue node/.style={circle,draw=blue, inner sep=2.5}}
    \tikzset{red node/.style={circle,draw=red, inner sep=2.5}}
    \tikzset{black node/.style={circle,draw=black, inner sep=2.5}}
    
	\node(1)[black node] at (2,0){$v_1$};
	\node(2)[black node] at (5,0){$v_2$};
	\node(3)[black node] at (3.5,2.5){$v_3$};
	\node(4)[black node] at (6.5,2.5){$v_4$};
    \node(5)[ghost node] at (3.5,-1){$G$};
	\draw (1) to  (2);
	\draw (1)  to (3);
	\draw (2) to (3);
	\draw (2)   to (4);
	\node(1)[black node] at (7,0){$\tcr{4}$};
	\node(2)[black node] at (10,0){$\tcr{3}$};
	\node(3)[black node] at (8.5,2.5){$\tcr{1}$};
	\node(4)[black node] at (11.5,2.5){$\tcr{2}$};
	\node(5)[ghost node] at (8.5,-1){$\tcr{4312}$};
	\draw (1)  to  (2);
	\draw (1)  to (3);
	\draw (2)  to (3);
	\draw (2)  to (4);
	\node(1)[black node] at (12,0){$\tcr{4}$};
	\node(2)[black node] at (15,0){$\tcr{2}$};
	\node(3)[black node] at (13.5,2.5){$\tcr{3}$};
	\node(4)[black node] at (16.5,2.5){$\tcr{1}$};
	\node(5)[ghost node] at (13.5,-1){$\tcr{4231}$};
	\draw (1)  to (2);
	\draw (1)  to (3);
	\draw (2)  to (3);
	\draw (2)  to (4);
	\node(1)[black node] at (17,0){$\tcr{1}$};
	\node(2)[black node] at (20,0){$\tcr{3}$};
	\node(3)[black node] at (18.5,2.5){$\tcr{2}$};
	\node(4)[black node] at (21.5,2.5){$\tcr{4}$};
	\node(5)[ghost node] at (18.5,-1){$\tcr{1234}$};
	\draw (1)  to (2);
	\draw (1)  to (3);
	\draw (2)  to (3);
	\draw (2)  to (4);
	\node(1)[black node] at (22,0){$\tcr{2}$};
	\node(2)[black node] at (25,0){$\tcr{3}$};
	\node(3)[black node] at (23.5,2.5){$\tcr{1}$};
	\node(4)[black node] at (26.5,2.5){$\tcr{4}$};
	\node(5)[ghost node] at (23.5,-1){$\tcr{2314}$};
	\draw (1)  to (2);
	\draw (1)  to (3);
	\draw (2)  to (3);
	\draw (2)  to (4);
\end{tikzpicture}
\end{center}
\end{example}

This new graph theoretic generalization recovers the results previously mentioned via the study of peaks on path and cycles graphs.
Motivated by this perspective, in the present work, we enumerate all labelings in a general graph with a fixed peak set by studying the possible vertices on which we can place the largest label. This process yields the following results:
\begin{enumerate}
\item Study peaks on $C_n$, cycle graph on $n$ vertices in Section \ref{sec:cycles} describing $|P(S;C_n)|$ as a sum of terms 
involving peak sets on path graphs in  Proposition~\ref{thm:cycles}.  
\item 
Generalize the methods in Section 2.1 and provide 
 Algorithm \ref{alg} for constructing 
all permutations in
$\PSG$ for an arbitrary graph $G$.

\item Consider graph joins and provide a collection of interesting special cases in Section \ref{sec:joins} that show that $|\PSG|$ often demonstrates factorial growth, and that the 
peak polynomials appearing in Equation \eqref{eq:peakpoly} are rare occurrences. 
\end{enumerate}

\section{Recursive Construction for Peaks on Graphs} \label{sec:recursion}

We begin our analysis of $P(S;G)$ by considering the case when $G=C_n$ is a cycle graph 
and we reduce this problem to studying peak sets on path graphs. Then in Subsection \ref{sec:construction}, we present Algorithm~\ref{alg}, which yields a construction of the set $\PSG$ for arbitrary graphs $G$. The reader may use the construction on cycle graphs as a concrete example or may skip to Subsection \ref{sec:construction} for Theorem~\ref{thm:alg_ya_orithm_ya} our main result. 



 \subsection{Cycles} \label{sec:cycles}

Let $C_n$ denote a cycle graph on $n$ vertices which we label by $v_1, v_2, \ldots, v_{n}$.
We say that a set $S =\{v_{i_1},v_{i_2},\ldots, v_{i_\ell}\}\subseteq V(C_n)$ is \textit{${C_n}$-admissible} if $P (S;C_n) \neq \emptyset$. For each $1\leq k\leq \ell$  let $\widehat{S}_{i_k}=S\setminus\{v_{i_k}\}$.
If $S$ is  $C_n$-admissible, then the label $n$ must be placed at a vertex $v_i$ in $S$. By removing the vertex $v_i$ and its incident edges from $C_n$, we obtain $C_n\setminus \{v_i\}$ a path graph on $n-1$ vertices whose peak set is $\widehat{S_i}$. 
We can now state our first set of results.

\begin{proposition}\label{thm:cycles}
If $S \subset V(C_n)$ is a $C_n$-admissible set, then  
 $ \left |P(S;C_{n})  \right | = \sum_{v_{i} \in S}  |P(\widehat{S_{i}}\,;\,C_n\setminus \{v_i\}) |.$
 \end{proposition}
 
Proposition~\ref{thm:cycles} follows from Theorem~\ref{thm:alg_ya_orithm_ya} in Subsection~\ref{sec:construction} , and hence  we omit the details.

 \begin{corollary}\label{cor:cyclespoly}
 If $S$ is a $C_n$-admissible set, then
$\left |P(S;C_{n}) \right | = 2^{n-|S|-1} \sum_{v_i \in S} p_{\widehat{S_{i}}}(n-1)$
  where $p_{\widehat{S_i}}$ denotes the peak polynomial of Equation \eqref{eq:peakpoly}.
 \end{corollary}
 
 \begin{proof}
 Using Proposition \ref{thm:cycles} and Equation \eqref{eq:peakpoly}, we get 
 \begin{align*} 
  \left |P(S;C_{n}) \right | & = \sum_{v_i \in S}  |P(\widehat{S_{i}}\,;\,C\setminus\{v_i\} ) |= \sum_{v_i \in S}  |P(\widehat{S_{i}}\,;\,P_{n-1} ) |\\ 
                                    &= \sum_{v_i \in S}  2^{(n-1)-|\widehat{S_{i}} |-1} p_{\widehat{S_{i}}}(n-1) 
                                     = 2^{n-|S|-1} \sum_{v_i \in S}  p_{\widehat{S_{i}}}(n-1).\qedhere
 \end{align*}
 \end{proof}

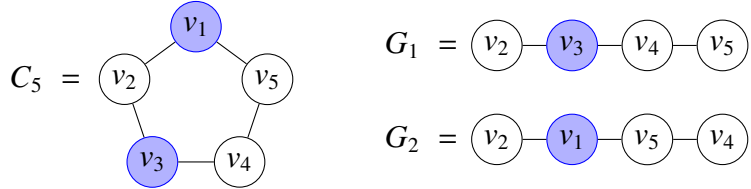
\begin{figure}[h!]
\begin{center}
\begin{tikzpicture}
	\tikzstyle{hollow node}=[draw=white]
    \tikzset{blue node/.style={circle,draw=blue, fill=blue!30, inner sep=2.5}}
    \tikzset{black node/.style={circle,draw=black, inner sep=2.5}}
     \node(c0) at (-2,.3){$C_5\;=$};
	\node(0)[blue node] at (0,1){$v_1$};
	\node(1)[black node] at (-.95,.30){$v_2$};
	\node(2)[blue node] at (-.58,-.80){$v_3$};
	\node(3)[black node] at (.58,-.80){$v_4$};
	\node(4)[black node] at (.95,.30){$v_5$};
	\draw (0)--(1);
	\draw (0)--(4);
	\draw (1)--(2);
	\draw (2)--(3);
	\draw (3)--(4);
    \node(a0) at (3,.75){$G_1\;=$};
	\node(a1)[black node] at (4,.75){$v_2$};
	\node(a2)[blue node] at (5,.75){$v_3$};
	\node(a3)[black node] at (6,.75){$v_4$};
	\node(a4)[black node] at (7,.75){$v_5$};
	\draw (a1)--(a2)--(a3)--(a4);
    \node(b0) at (3,-.5){$G_2\;=$};
	\node(b1)[black node] at (4,-.5){$v_2$};
	\node(b2)[blue node] at (5,-.5){$v_1$};
	\node(b3)[black node] at (6,-.5){$v_5$};
	\node(b4)[black node] at (7,-.5){$v_4$};
	\draw (b1)--(b2)--(b3)--(b4);	
\end{tikzpicture}
\end{center}
\caption{Cycle graph on 5 vertices and path graphs $G_1$ and $G_2$ obtained from removing vertices $v_3$ and $v_1$ from $C_5$, respectively.}\label{fig:cycle5}
\end{figure}
\begin{example}
Consider the graph $C_{5}$ in Figure \ref{fig:cycle5}. If $S = \{v_1,v_3\}\subseteq V(C_{5})$ 
 then the sets $\widehat{S_1}=\{v_3\}$ and $\widehat{S_3}=\{v_1\}$. One can verify that 
 \begin{align*}
 P(\widehat{S_{1}}=\{v_3\}\,;\,G_1)&=
 P(\widehat{S_{3}}=\{v_1\}\,;\,G_2)=
 \{ 1324,2314,1432,1423,2431,3421,3412\}
 \end{align*}
 where $G_1$ and $G_2$ are isomorphic to $P_4$ as shown in Figure \ref{fig:cycle5}.
Proposition~\ref{thm:cycles} yields $$ |P(S;C_{5})| =  | P(\widehat{S_{1}}\,;G_1)| +| P(\widehat{S_{3}}\,;G_2)|=16$$ and one can verify that in fact
\begin{align*}
P(S;C_{5})&=\left\{\begin{tabular}{cccccccc}
31542,& 32541,& 41523,& 41532,& 42513,& 42531,& 43512,& 43521,\\ 51324,& 51423,& 51432,& 52314,& 52413,& 52431,& 53412,& 53421
\end{tabular}
\right
\}.
\end{align*}
\end{example}


\subsection{General constructive algorithm for graphs} \label{sec:construction}
In this section we describe a recursive algorithm for constructing the set 
$\PSG$ consisting of all labelings of the vertices a graph $G$ with a given peak 
set $S$.   

We begin by setting the following notation.
 For any vertex $v \in V(G)$, the \emph{neighborhood} of $v$, denoted $N_G(v)$ 
 is the set $N_G(v):=\{ w \in V(G): \{v,w\} \text{ is an edge in $G$} \}.$  For any $S \subseteq V(G)$, we let $N_G(S)$ be the neighborhood set of $S$, namely $N_G(S)=\cup_{v\in S} N_G(v)$. As is standard, we say $S$ is  an \textit{independent set} if no vertex in $S$ is in the neighborhood of any other vertex in $S$, i.e. $S\cap N_G(S) =\emptyset$.

\begin{algorithm}
    \SetKwInOut{Input}{Input}
    \SetKwInOut{Output}{Output}

    \underline{function  GraphPeakSetAlgorithm}$(G,S,L,P)$\;
    \Input{Graph $G=(V,E)$ and Peak Set $S\subseteq V$,  $L$ a set of vertices which contains the leaves and isolated vertices of $G$, and a list $P$.}
    \For{$v \in S \cup ({L\setminus N_G(S)})$}{
    \If{$v \in S$}{ 
    $S \leftarrow S \backslash \{v\}$\;
    }
   { $L\leftarrow (L \cup N_G(v)) \cap (V(G)\backslash \{v\})$}\; $P[v] = |V|$\; 
    \If{$|V(G \backslash \{v\})|>0$}{
    \mbox{GraphPeakSetAlgorithm}$(G \backslash \{v\}, S, L, P)$\;
    }
    \If{$|V(G \backslash \{v\})|=0$}{
    \mbox{return} $P$
    }
    }    
    \caption{Graph Peak Set Algorithm}\label{alg}    
\end{algorithm}


Before verifying  Algorithm~\ref{alg} we present a graphical example that includes every possible choice in line 2 of the algorithm, as well as the algorithm's output.
\begin{example}
Consider the graph $G$ in Figure \ref{fig:algorithm}, and let $S=\{ v_1 \}$. We apply Algorithm~\ref{alg} to the graph $G$ with vertices $v_1,v_2,v_3,v_4$. At every iteration of the algorithm, we label a vertex with the largest available number and then remove that vertex from the graph, but for illustrative purposes we color the removed vertices instead of physically removing them from $G$. Figure~\ref{fig:algorithm} provides the inputs $S,L,P$ for the first iteration of Algorithm~\ref{alg}, as well as the final output $P$, which records the labeled graphs.
Writing each labeling of $G$ as a permutation, with the label of $v_i$ as the image of $i$, we obtain 
$\PSG = \{ 4321, 4312, 4231,  4132, 4123, 4213 , 3124, 3214 \} .$
\end{example}

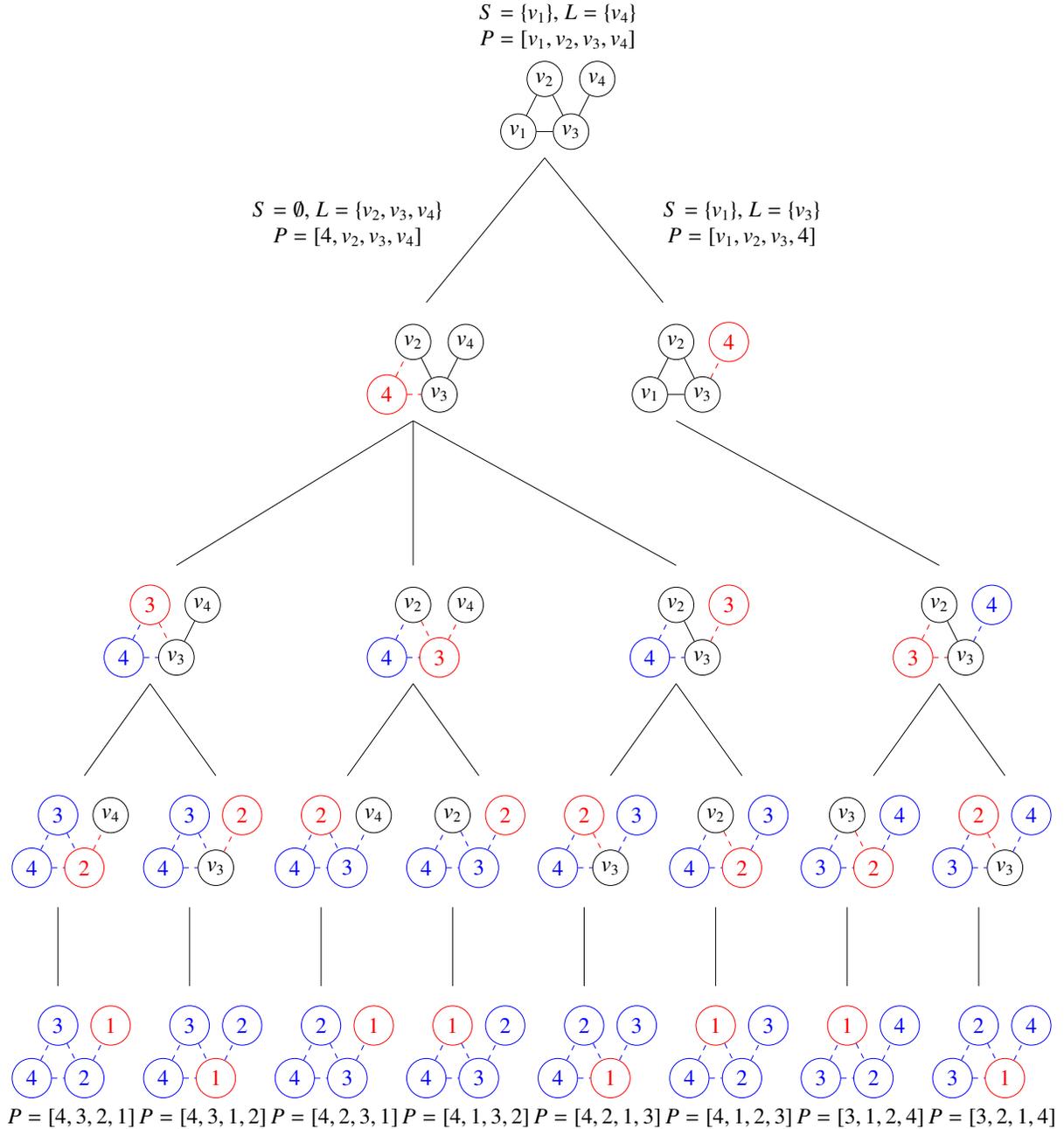
\begin{figure}[h]
\centering
\scalebox{.8}{
\begin{tikzpicture}
	\tikzstyle{hollow node}=[draw=white]
    \tikzset{black node/.style={circle,draw=black, inner sep=2.5}}
    \tikzset{red node/.style={circle,draw=red}}
 \tikzset{blue node/.style={circle,draw=blue}}

\node at (10,10.25) {$S=\{v_1\}$, $L=\{v_4\}$};
\node at (10,9.75) {$P=[v_1,v_2,v_3,v_4]$};

	\node(51a)[black node] at (9.25,8){$v_1$};
	\node(52a)[black node] at (9.75,9){$v_2$};
	\node(53a)[black node] at (10.25,8){$v_3$};
	\node(54a)[black node] at (10.75,9){$v_4$};
	\draw (51a)--(52a);
	\draw (51a)--(53a);
	\draw (52a)--(53a);
	\draw (53a)--(54a);
\draw(9.75,7.5)--(7.5,4.75);
\draw(9.75,7.5)--(12,4.75);

\node at (6,6.5) {$S=\emptyset$, $L=\{v_2, v_3, v_4\}$};
\node at (6,6) {$P=[4,v_2,v_3,v_4]$};
\node at (13.5,6.5) {$S=\{v_1\}$, $L=\{v_3\}$};
\node at (13.5,6) {$P=[v_1,v_2,v_3,4]$};

\node(51a)[red node] at (6.75,3){$\textcolor{red}{4}$};
	\node(52a)[black node] at (7.25,4){$v_2$};
	\node(53a)[black node] at (7.75,3){$v_3$};
	\node(54a)[black node] at (8.25,4){$v_4$};
	\draw[dashed, red] (51a)--(52a);
	\draw[dashed, red] (51a)--(53a);
	\draw (52a)--(53a);
	\draw (53a)--(54a);
	
\node(51a)[black node] at (11.75,3){$v_1$};
	\node(52a)[black node] at (12.25,4){$v_2$};
	\node(53a)[black node] at (12.75,3){$v_3$};
	\node(54a)[red node] at (13.25,4){$\textcolor{red}{4}$};
	\draw (51a)--(52a);
	\draw (51a)--(53a);
	\draw (52a)--(53a);
	\draw[dashed, red] (53a)--(54a);


\draw(7.25,2.5)--(2.75,-.25);
\draw(7.25,2.5)--(7.25,-.25);
\draw(7.25,2.5)--(12.25,-.25);	
\draw(12.25,2.5)--(17.25,-.25);

\node(51a)[blue node] at (1.75,-2){$\textcolor{blue}{4}$};
	\node(52a)[red node] at (2.25,-1){$\textcolor{red}{3}$};
	\node(53a)[black node] at (2.75,-2){$v_3$};
	\node(54a)[black node] at (3.25,-1){$v_4$};
	\draw[dashed, blue] (51a)--(52a);
	\draw[dashed, blue] (51a)--(53a);
	\draw[dashed, red] (52a)--(53a);
	\draw (53a)--(54a);

\node(51a)[blue node] at (6.75,-2){$\textcolor{blue}{4}$};
	\node(52a)[black node] at (7.25,-1){$v_2$};
	\node(53a)[red node] at (7.75,-2){$\textcolor{red}{3}$};
	\node(54a)[black node] at (8.25,-1){$v_4$};
	\draw[dashed, blue] (51a)--(52a);
	\draw[dashed, blue] (51a)--(53a);
	\draw[dashed, red] (52a)--(53a);
	\draw[dashed, red] (53a)--(54a);	

\node(51a)[blue node] at (11.75,-2){$\textcolor{blue}{4}$};
	\node(52a)[black node] at (12.25,-1){$v_2$};
	\node(53a)[black node] at (12.75,-2){$v_3$};
	\node(54a)[red node] at (13.25,-1){$\textcolor{red}{3}$};
	\draw[dashed, blue] (51a)--(52a);
	\draw[dashed, blue] (51a)--(53a);
	\draw (52a)--(53a);
	\draw[dashed, red] (53a)--(54a);	
	
\node(51a)[red node] at (16.75,-2){$\textcolor{red}{3}$};
	\node(52a)[black node] at (17.25,-1){$v_2$};
	\node(53a)[black node] at (17.75,-2){$v_3$};
	\node(54a)[blue node] at (18.25,-1){$\textcolor{blue}{4}$};
	\draw[dashed, red] (51a)--(52a);
	\draw[dashed, red] (51a)--(53a);
	\draw (52a)--(53a);
	\draw[dashed,blue] (53a)--(54a);

\draw(2.25,-2.5)--(1,-4.25);
\draw(2.25,-2.5)--(3.5,-4.25);
\draw(7.25,-2.5)--(6,-4.25);
\draw(7.25,-2.5)--(8.5,-4.25);
\draw(12.25,-2.5)--(11,-4.25);
\draw(12.25,-2.5)--(13.5,-4.25);
\draw(17.25,-2.5)--(16,-4.25);
\draw(17.25,-2.5)--(18.5,-4.25);

	\node(51a)[blue node] at (0,-6){$\textcolor{blue}{4}$};
	\node(52a)[blue node] at (.5,-5){$\textcolor{blue}{3}$};
	\node(53a)[red node] at (1,-6){$\textcolor{red}{2}$};
	\node(54a)[black node] at (1.5,-5){$v_4$};
	\draw[dashed, blue] (51a)--(52a);
	\draw[dashed,blue] (51a)--(53a);
	\draw[dashed,blue] (52a)--(53a);
	\draw[dashed,red] (53a)--(54a);
	
	\node(51b)[blue node] at (2.5,-6){$\textcolor{blue}{4}$};
	\node(52b)[blue node] at (3,-5){$\textcolor{blue}{3}$};
	\node(53b)[black node] at (3.5,-6){$v_3$};
	\node(54b)[red node] at (4,-5){$\textcolor{red}{2}$};
	\draw[dashed,blue] (51b)--(52b);
	\draw[dashed,blue] (51b)--(53b);
	\draw[dashed,blue] (52b)--(53b);
	\draw[dashed,red] (53b)--(54b);

	\node(51c)[blue node] at (5,-6){$\textcolor{blue}{4}$};
	\node(52c)[red node] at (5.5,-5){$\textcolor{red}{2}$};
	\node(53c)[blue node] at (6,-6){$\textcolor{blue}{3}$};
	\node(54c)[black node] at (6.5,-5){$v_4$};
	\draw[dashed,blue] (51c)--(52c);
	\draw[dashed,blue]  (51c)--(53c);
	\draw[dashed,blue]  (52c)--(53c);
	\draw[dashed,blue]  (53c)--(54c);

	\node(51d)[blue node] at (7.5,-6){$\textcolor{blue}{4}$};
	\node(52d)[black node] at (8,-5){$v_2$};
	\node(53d)[blue node] at (8.5,-6){$\textcolor{blue}{3}$};
	\node(54d)[red node] at (9,-5){$\textcolor{red}{2}$};
		\draw[dashed, blue]  (51d)--(52d);
		\draw[dashed, blue]  (51d)--(53d);
		\draw[dashed, blue]  (52d)--(53d);
		\draw[dashed, blue]  (53d)--(54d);
	
	\node(51e)[blue node] at (10,-6){$\textcolor{blue}{4}$};
	\node(52e)[red node] at (10.5,-5){$\textcolor{red}{2}$};
	\node(53e)[black node] at (11,-6){$v_3$};
	\node(54e)[blue node] at (11.5,-5){$\textcolor{blue}{3}$};
		\draw[dashed, blue]  (51e)--(52e);
		\draw[dashed, blue]  (51e)--(53e);
		\draw[dashed, red]  (52e)--(53e);
		\draw[dashed, blue]  (53e)--(54e);

	\node(51f)[blue node] at (12.5,-6){$\textcolor{blue}{4}$};
	\node(52f)[black node] at (13,-5){$v_2$};
	\node(53f)[red node] at (13.5,-6){$\textcolor{red}{2}$};
	\node(54f)[blue node] at (14,-5){$\textcolor{blue}{3}$};
		\draw[dashed, blue]  (51f)--(52f);
		\draw[dashed, blue]  (51f)--(53f);
		\draw[dashed, red]  (52f)--(53f);
		\draw[dashed, blue]  (53f)--(54f);

	\node(51g)[blue node] at (15,-6){$\textcolor{blue}{3}$};
	\node(52g)[black node] at (15.5,-5){$v_3$};
	\node(53g)[red node] at (16,-6){$\textcolor{red}{2}$};
	\node(54g)[blue node] at (16.5,-5){$\textcolor{blue}{4}$};
		\draw[dashed, blue]  (51g)--(52g);
		\draw[dashed, blue]  (51g)--(53g);
		\draw[dashed, red]  (52g)--(53g);
		\draw[dashed, blue]  (53g)--(54g);

	\node(51h)[blue node] at (17.5,-6){$\textcolor{blue}{3}$};
	\node(52h)[red node] at (18,-5){$\textcolor{red}{2}$};
	\node(53h)[black node] at (18.5,-6){$v_3$};
	\node(54h)[blue node] at (19,-5){$\textcolor{blue}{4}$};
		\draw[dashed, blue]  (51h)--(52h);
		\draw[dashed, blue]  (51h)--(53h);
		\draw[dashed, red]  (52h)--(53h);
		\draw[dashed, blue]  (53h)--(54h);

\draw(.5,-8.25)--(.5,-6.75);
\draw(3,-8.25)--(3,-6.75);
\draw(5.5,-8.25)--(5.5,-6.75);
\draw(8,-8.25)--(8,-6.75);
\draw(10.5,-8.25)--(10.5,-6.75);
\draw(13,-8.25)--(13,-6.75);
\draw(15.5,-8.25)--(15.5,-6.75);
\draw(18,-8.25)--(18,-6.75);

	\node(51a)[blue node] at (0,-10){$\textcolor{blue}{4}$};
	\node(52a)[blue node] at (.5,-9){$\textcolor{blue}{3}$};
	\node(53a)[blue node] at (1,-10){$\textcolor{blue}{2}$};
	\node(54a)[red node] at (1.5,-9){$\textcolor{red}{1}$};
		\draw[dashed, blue]  (51a)--(52a);
		\draw[dashed, blue]  (51a)--(53a);
		\draw[dashed, blue]  (52a)--(53a);
		\draw[dashed, blue]  (53a)--(54a);
	
	\node(51b)[blue node] at (2.5,-10){$\textcolor{blue}{4}$};
	\node(52b)[blue node] at (3,-9){$\textcolor{blue}{3}$};
	\node(53b)[red node] at (3.5,-10){$\textcolor{red}{1}$};
	\node(54b)[blue node] at (4,-9){$\textcolor{blue}{2}$};
		\draw[dashed, blue]  (51b)--(52b);
		\draw[dashed, blue]  (51b)--(53b);
		\draw[dashed, blue]  (52b)--(53b);
		\draw[dashed, blue]  (53b)--(54b);

	\node(51c)[blue node] at (5,-10){$\textcolor{blue}{4}$};
	\node(52c)[blue node] at (5.5,-9){$\textcolor{blue}{2}$};
	\node(53c)[blue node] at (6,-10){$\textcolor{blue}{3}$};
	\node(54c)[red node] at (6.5,-9){$\textcolor{red}{1}$};
		\draw[dashed, blue]  (51c)--(52c);
		\draw[dashed, blue]  (51c)--(53c);
		\draw[dashed, blue]  (52c)--(53c);
		\draw[dashed, blue]  (53c)--(54c);

	\node(51d)[blue node] at (7.5,-10){$\textcolor{blue}{4}$};
	\node(52d)[red node] at (8,-9){$\textcolor{red}{1}$};
	\node(53d)[blue node] at (8.5,-10){$\textcolor{blue}{3}$};
	\node(54d)[blue node] at (9,-9){$\textcolor{blue}{2}$};
		\draw[dashed, blue]  (51d)--(52d);
		\draw[dashed, blue]  (51d)--(53d);
		\draw[dashed, blue]  (52d)--(53d);
		\draw[dashed, blue]  (53d)--(54d);
	
	\node(51e)[blue node] at (10,-10){$\textcolor{blue}{4}$};
	\node(52e)[blue node] at (10.5,-9){$\textcolor{blue}{2}$};
	\node(53e)[red node] at (11,-10){$\textcolor{red}{1}$};
	\node(54e)[blue node] at (11.5,-9){$\textcolor{blue}{3}$};
		\draw[dashed, blue]  (51e)--(52e);
		\draw[dashed, blue]  (51e)--(53e);
		\draw[dashed, blue]  (52e)--(53e);
		\draw[dashed, blue]  (53e)--(54e);

	\node(51f)[blue node] at (12.5,-10){$\textcolor{blue}{4}$};
	\node(52f)[red node] at (13,-9){$\textcolor{red}{1}$};
	\node(53f)[blue node] at (13.5,-10){$\textcolor{blue}{2}$};
	\node(54f)[blue node] at (14,-9){$\textcolor{blue}{3}$};
		\draw[dashed, blue]  (51f)--(52f);
		\draw[dashed, blue]  (51f)--(53f);
		\draw[dashed, blue]  (52f)--(53f);
		\draw[dashed, blue]  (53f)--(54f);

	\node(51g)[blue node] at (15,-10){$\textcolor{blue}{3}$};
	\node(52g)[red node] at (15.5,-9){$\textcolor{red}{1}$};
	\node(53g)[blue node] at (16,-10){$\textcolor{blue}{2}$};
	\node(54g)[blue node] at (16.5,-9){$\textcolor{blue}{4}$};
		\draw[dashed, blue]  (51g)--(52g);
		\draw[dashed, blue]  (51g)--(53g);
		\draw[dashed, blue]  (52g)--(53g);
		\draw[dashed, blue]  (53g)--(54g);

	\node(51h)[blue node] at (17.5,-10){$\textcolor{blue}{3}$};
	\node(52h)[blue node] at (18,-9){$\textcolor{blue}{2}$};
	\node(53h)[red node] at (18.5,-10){$\textcolor{red}{1}$};
	\node(54h)[blue node] at (19,-9){$\textcolor{blue}{4}$};
	\draw[dashed, blue] (51h)--(52h);
	\draw[dashed, blue] (51h)--(53h);
	\draw[dashed, blue] (52h)--(53h);
	\draw[dashed, blue] (53h)--(54h);
\node at (.75,-10.75) {$P=[4,3,2,1]$};
\node at (3.25,-10.75) {$P=[4,3,1,2]$};
\node at (5.75,-10.75) {$P=[4,2,3,1]$};
\node at (8.25,-10.75) {$P=[4,1,3,2]$};
\node at (10.75,-10.75) {$P=[4,2,1,3]$};
\node at (13.25,-10.75) {$P=[4,1,2,3]$};
\node at (15.75,-10.75) {$P=[3,1,2,4]$};
\node at (18.25,-10.75) {$P=[3,2,1,4]$};

\end{tikzpicture}
}
\caption{Application of Algorithm~\ref{alg} on a small graph $G$.}\label{fig:algorithm}
\end{figure}

The next definition plays an important role in the proof of our main result Theorem~\ref{thm:alg_ya_orithm_ya}.
\begin{definition} \label{def:upsetting_vic_reiner}
Let $V_{<2}(G)$ denote the set of vertices in $G$ of degree less than 2, and let $L$ be a set with $V_{<2}(G) \subseteq L \subseteq V(G)$.
Define $P(S,G,L)$ to be the \textit{set of labelings of $G$ with peak set $S$ or peak set $S'$} with $S \subseteq S' \subseteq S \cup (L \setminus N_{G}(S))$.
\end{definition}

If $L = V_{<2}(G)$ then the only potentially admissible peak set $S'$ satisfying
$S \subseteq S' \subseteq S \cup (L \setminus N_{G}(S))$ is $S$ itself because none of the elements of $L$ can be peaks.  In this case $P(S,G,L) =P(S;G)$.

\begin{theorem} \label{thm:alg_ya_orithm_ya}
Let $L \subset V(G)$ be a set of vertices and suppose $V_{< 2}(G) \subseteq L$.  
Then the output  $GraphPeakSetAlgorithm(G,S,L,P)$ is $P(S,G,L)$.  
\end{theorem}

The proof of Theorem~\ref{thm:alg_ya_orithm_ya} results from the following technical lemmas.

\begin{lemma} \label{lemma:containment1}
The set $P(S,G,L)$ is a subset of the output  $GraphPeakSetAlgorithm(G,S,L,P)$. \end{lemma}

\begin{proof}
Let $\lab$ be a labeling in $P(S,G,L)$.  We prove the result by induction on $|V(G)|$.
Firstly, if $V(G)=\{v\}$ then the only possible peak set is $S=\emptyset$. Since $L$ must contain any vertex of degree less than 2, the only possible $L$ is $\{v\}$. Then $P(S,G,L)$ consists of one labeling, hence $P(S,G,L)=\{[1]\}$.  Since $G$ has one vertex $v$, $GraphPeakSetAlgorithm(G,S,L,P)$ picks $v \in L\setminus\{N_G(S)\}$ in line 2, updates $L$ in lines 3-6, labels $v $ with $1$ in line $7$ and then returns the only labeling $P[v]=1$.   Thus $P(S,G,L) = GraphPeakSetAlgorithm(G,S,L,P)$ in the base case.

Now we suppose that the statement is true for all admissible peak sets $S \subseteq V(G') $ and any set $ L'$ with $ V_{<2}(G') \subseteq L' \subseteq V(G')$ on graphs $G'$ with $1 \leq  |V(G')| <n$.
Let $G$ be any graph with $|V(G)| =n$,  $S$ be an admissible peak set on $G$, and $L$ be a set with $ V_{<2}(G) \subseteq L \subseteq V(G)$.
Let $\lab$ be a labeling in $P(S,G,L)$. There are two cases to consider depending on where the label $n$ can appear in $\lab$: either $n$ labels a vertex in $S$ or $n$ labels a vertex in $L$ with $V_{<2}(G) \subseteq L\subseteq V(G)$. We will show that
$\lab$ is an output of $GraphPeakSetAlgorithm(G,S,L,P)$.\\ 

\noindent
\textbf{Case 1:} If $n$ labels a vertex $v$ in $S$, then when running $GraphPeakSetAlgorithm(G,S,L,P)$, in line 2 we consider $v\in S$ as the chosen vertex. Then, in lines 3-6 we update $S$ to be $S'=S\setminus\{v\}$ and $L$ to be $L'=(L \cup N_G(v)) \cap (V(G) \setminus \{v\})$. In line 7, we label $v$ by $P[v]=n=|V(G)|$ and then, in lines 8-10, we call
$GraphPeakSetAlgorithm(G \setminus \{v\}, S', L', P)$. It is now enough to show that the labeling $\lab'$, obtained from deleting $v$ in $\lab$, is an output of $GraphPeakSetAlgorithm(G\setminus\{v\}, S',L',P)$ in line 9. We do this by showing that $\lab'$ is in $P(S',G\setminus\{v\},L')$, which is a subset of $GraphPeakSetAlgorithm(G\setminus\{v\}, S',L',P)$ by our induction hypothesis. 

Note that by removing $v$ from $\lab$ we might create peaks at vertices in $N_G(v)$, and we keep all other peaks of $\lab$. Thus, the peak set $S''$ of $\lab'$ satisfies 
\begin{equation}\label{eq:psgsubset} S \setminus \{v\} \subseteq S'' \subseteq (S \setminus \{v\}) \cup (L \setminus N_{G }(S )) \cup N_{G} (v) .\end{equation}
Going back to the definition of $P(S',G\setminus\{v\},L')$, to show $\lab' \in P(S',G\setminus\{v\},L')$, we must show that the peak set $S''$ of $\lab'$ satisfies 
\begin{equation}\label{eq:psgsubset2}
S\setminus\{v\} \subseteq S'' \subseteq \bigg(S\setminus \{v\}\bigg) \cup \bigg[\bigg((L \cup N_G(v)) \cap (V(G) \setminus \{v\})\bigg) \setminus N_{G\setminus\{v\}}(S\setminus\{v\})\bigg].\end{equation}
Note that the only elements in the set to the right of the second containment in \eqref{eq:psgsubset} that are not in the set to the right of the second containment in \eqref{eq:psgsubset2} are the elements in $N_G(v) \cap N_{G\setminus\{v\}}(S\setminus \{v\})$. Since these elements are in $N_{G\setminus\{v\}}(S\setminus \{v\})$, they are connected to peaks, hence cannot be peaks themselves. Thus, they cannot be in $S''$, which proves the containment in \eqref{eq:psgsubset2}.\\

\noindent
\textbf{Case 2:} If $n$ labels a vertex $v$ in $L\setminus N_G(S)$, then when running $GraphPeakSetAlgorithm(G,S,L,P)$, in line 2 we consider $v \in L\setminus N_G(S)$ as the chosen vertex. Then, we skip lines 3-5 and update $L$ to be $L'=(L \cup N_G(v)) \cap (V(G) \setminus \{v\})$ in line 6. In line 7, we label $v$ by $P[v]=n=|V(G)|$ and then, in lines 8-10, we call $GraphPeakSetAlgorithm(G \setminus \{v\}, S, L', P)$. Similar to the previous case, it is now enough to show that the labeling $\lab'$, obtained from deleting $v$ in $\lab$, is an output of $GraphPeakSetAlgorithm(G\setminus\{v\}, S,L',P)$ in line 9. We do this by showing that $\lab'$ belongs to $P(S,G\setminus\{v\},L')$.

Note that by ignoring $v$ in $\lab$ we obtain a labeling $\lab'$ on $G \setminus \{v\}$ with peak set $S''$ where
\begin{equation}\label{eq:pgssubset3}S\subseteq S'' \subseteq S \cup (L\setminus N_G(S)) \cup N_{G}(v).\end{equation}
To prove $\lab' \in P(S,G\setminus\{v\},L')$, we must show
\begin{equation}\label{eq:pgssub4} 
S\subseteq S'' \subseteq S \cup \bigg[\bigg((L \cup N_G(v)) \cap (V(G) \setminus \{v\})\bigg) \setminus N_{G\setminus\{v\}}(S)\bigg]. \end{equation}
Note that the only elements in the set to the right of the second containment in \eqref{eq:pgssubset3} that are not in the set to the right of the second containment in \eqref{eq:pgssub4} are the elements in $N_G(v) \cap N_{G\setminus\{v\}}(S)$. Since these elements are in $N_{G\setminus\{v\}}(S)$, they are connected to peaks, hence cannot be peaks themselves. Thus, they cannot be in $S''$, which proves the containment in~\eqref{eq:pgssub4}.
\end{proof}

\begin{lemma} \label{lemma:containment2}
The set consisting of the output of $GraphPeakSetAlgorithm(G,S,L,P)$ is a subset of $P(S,G,L)$.  
\end{lemma}

\begin{proof}
We will show that any output of $GraphPeakSetAlgorithm(G,S,L,P)$, $\lab$, is also an element of $P(G,S,L)$ by induction on $|V(G)|$.
Let $|V(G)|=1$, then the only possible peak set is $S=\emptyset$. Since $L$ must contain any vertex of degree less than 2, the only possible $L$ is $\{v\}$. Hence $P(S,G,L)$ consists of one labeling $P(S,G,L)=\{[1]\}$.  Since $G$ has one vertex $v$, $GraphPeakSetAlgorithm(G,S,L,P)$ picks $v \in L\setminus\{N_G(S)\}$ in line 2, updates $L$ in lines 3-6, labels $v $ with $1$ in line $7$ and then returns the only labeling $P[v]=1$.   Thus $P(S,G,L) = GraphPeakSetAlgorithm(G,S,L,P)$ in the base case.

Let $G$ be a graph with $|V(G)|=n$, $S$ be an admissible peak set, and $L$ be a set satisfying $ V_{<2}(G) \subseteq L \subseteq V(G)$. Let $\lab$ be an output of $GraphPeakSetAlgorithm(G, S, L, P)$. Note that in step 2 of $GraphPeakSetAlgorithm(G,S,L,P)$, a certain vertex of $v \in S \cup (L\setminus\{N_G(S)\})$ is chosen. In steps 3-6, some elements are added to the sets $S$ and $L$. Let $S', L'$ denote the resulting sets. In step 7, the vertex $v$ is labeled by $n$, i.e., $\lab[v]=n$. In steps 8-10, we run $GraphPeakSetAlgorithm(G\backslash \{v\}, S', L', P)$. Let $\lab'$ be the labeling we obtain by removing $v$ from $\lab$. Then by the induction hypothesis, $\lab'$ is an element of $P(S',G\backslash \{v\},L')$, i.e., $\lab'$ has peak set $\overline{S'}$ with $S' \subseteq \overline{S'}\subseteq S' \cup (L' \setminus N_G(S'))$. To show $\lab$ is an element of $P(G,S,L)$, we must show that when inserting $v$ into $\lab'$ with label $n$, we get a labeling with peak set $\overline{S}$ where $S \subseteq \overline{S}\subseteq S \cup (L \setminus N_G(S))$. There are two cases to consider. \\

\noindent
\textbf{Case 1:} If the chosen vertex $v$ is in $S$, then $S'=S\backslash\{v\}$ and $L'=(L\cup N_G(v))\cap (V(G)\backslash \{v\})$. Thus, $\lab'$ has peak set $\overline{S'}$ with 
\begin{equation}\label{eq:peakoflab}
S\backslash\{v\} \subseteq \overline{S'}\subseteq (S\backslash\{v\}) \cup \bigg[\bigg((L\cup N_G(v))\cap (V(G)\backslash \{v\})\bigg) \setminus N_G(S\backslash\{v\})\bigg].\end{equation}
By inserting the previously removed label $n$, indexed by $v$, back into $\lab'$, we now create a peak at $v$, and remove any previous peak in vertices that are neighbors of $v$. Using the containment relations in \eqref{eq:peakoflab}, we get that the peak set $\overline{S}$ of $\lab$ satisfies
\[ S\subseteq \overline{S} \subseteq S\cup \bigg( L\setminus N_G(S\setminus \{v\})\bigg) .\]
Since $v$ is a peak, we are certain no vertex in $N_G(v)$ is also a peak, thus we can further say $S\subseteq \overline{S} \subseteq S\cup ( L\setminus N_G(S))$, which completes the first case.\\

\noindent
\textbf{Case 2:} We now consider when the chosen vertex $v$ is in $L\setminus N_G(S)$. Then $S'=S$ and $L'=(L\cup N_G(v))\cap (V(G)\backslash \{v\})$. Thus, $\lab'$ has peak set $\overline{S'}$ with 
\begin{equation}\label{eq:peakoflab2}
S \subseteq \overline{S'}\subseteq S \cup \bigg[\bigg((L\cup N_G(v))\cap (G\backslash \{v\})\bigg) \setminus N_G(S)\bigg].\end{equation}
By inserting the previously removed label $n$, indexed by $v$, back into $\lab'$, we either create a peak at $v$ (if $deg_G(v)\geq 2$) or do not create a peak  at $v$ (if $deg_G(v)<2$). In either case, we also remove any previous peaks in vertices that are neighbors of $v$. Using the containment relations in \eqref{eq:peakoflab2}, we get that the peak set $\overline{S}$ of $\lab$ satisfies 
\[S\subseteq \overline{S} \subseteq S\cup(L\backslash N_G(S)).\]
We have shown that regardless of the choice of $v$ in line 2, the output $\lab$ has peak set $\overline{S}$ where 
$S\subseteq \overline{S} \subseteq S\cup(L\backslash N_G(S))$; in other words $\lab$ is in $P(S,G,L)$.
\end{proof}


\section{Graph Joins}\label{sec:joins}
In this section, we study the relationship 
between  the peak set of the join of two graphs in terms of the peak sets of the constituent graphs.  
We prove three main results related to the peak sets of graph joins: Proposition \ref{prop:oneside} on the joins of two arbitrary graphs, Proposition \ref{theorem:nullgraphjoin} on the join of an arbitrary graph with a complete graph, and Proposition \ref{prop:complete} on the join of an arbitrary graph with a null graph.

We recall that the \emph{join} of $G_1$ and $G_2$, 
denoted $G_1 \vee G_2$, 
has vertex set $V(G_1) \cup V(G_2)$ and edge set $E(G_1) \cup E(G_2) \cup \{\{v_1,v_2\} \ | \ v_1 \in V(G_1), \ v_2 \in V(G_2)\}$. 
For example, let $\overline{K_n}$ be the null graph given by an independent set on $n$ vertices,
and $K_n$ be the complete graph on $n$ vertices. A star graph on $n$ vertices is 
$K_1 \vee \overline{K_{n-1}}$ and a complete bipartite graph is $K_{m,n} = \overline{K_m} \vee \overline{K_n}$ 
(see Figure \ref{fig:examples of joins}).

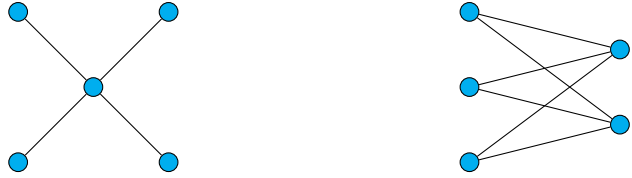
\begin{figure}[h!]
\centering
\begin{tikzpicture}[scale=1]
   \tikzstyle{hollow node}=[draw=white]
   \tikzset{blue node/.style={circle,fill=cyan,draw=black, inner sep=2.5}
}
   \node(0)[blue node] at (-7,3){};
   \node(1)[blue node] at (-8,4){};
   \node(2)[blue node] at (-8,2){};
   \node(3)[blue node] at (-6,2){};
   \node(4)[blue node] at (-6,4){};
   
   \draw(0)--(1);
   \draw(0)--(2);
   \draw(0)--(3);
   \draw(0)--(4);

   \node(5)[blue node] at (-2,4){};
   \node(6)[blue node] at (-2,3){};
   \node(7)[blue node] at (-2,2){};
   \node(8)[blue node] at (0,3.5){};
   \node(9)[blue node] at (0,2.5){};   
      
   \draw(5)--(8);
   \draw(5)--(9);
   \draw(6)--(8);
   \draw(6)--(9);
   \draw(7)--(8);
   \draw(7)--(9);
\end{tikzpicture}
\caption{The graph on the left is $K_1 \vee \overline{K_4}$, a star graph on $5$ vertices. The graph on the right is $K_{3,2}=\overline{K_3} \vee \overline{K_2}$.}\label{fig:examples of joins}
\end{figure}

We first show that peak sets in $G_1 \vee G_2$ are completely contained in either 
$V(G_1)$ or $V(G_2)$.  

\begin{proposition}\label{prop:oneside}
Let $G_1$ and $G_2$ be arbitrary graphs with vertex sets $V_1$ and $V_2$ respectively. Suppose that
$|V_1|>1$ and $|V_2|>1$.  Let $S$ be a nonempty ($G_1 \vee G_2)$-admissible peak set. Then $S \subseteq V_1$ or $S \subseteq V_2$.  
\end{proposition}

\begin{proof}
In any labeling of $G_1 \vee G_2$, there is some vertex $v$ labeled by the number $N=|V_1 \cup V_2|$. 
Assume the vertex $v$ is in $V_1$. 
Then no vertex of $V_2$ can be a peak because $v$ is adjacent to every vertex in $V_2$, so $S \subseteq V_1$. Similarly, if $v\in V_2$, then $S \subseteq V_2$.
\end{proof}

The next two results give explicit formulas for the number of labelings with a given peak set $S$ on 
joins of an arbitary graph $G$ with either the null graph $\ol{K_n}$ or the complete graph $K_n$.
 
\begin{proposition} \label{theorem:nullgraphjoin}
 Let $S\subseteq V(\overline{K_n})$ be nonempty, $G$ be any graph with $|V(G)|>1$, and $G' = \overline{K_n} \vee G$.   
 Then the set $S$ is admissible in $G'$ with
 \[ |P(S;G')|  = |S|! \cdot |V(G)| \cdot (|V(G')|-|S|-1)! \]

\end{proposition}
\begin{proof}
Let $m=|V(G')|$. First we claim that the vertices in $S$ must be labeled by the set 
\[\mathcal{I}= \{m,m-1,\ldots,m-n-1\}.\] Otherwise there are two possible cases: (1) Some vertex in $V(G)$ will be labeled by an element in $\mathcal{I}$ while some element of $S$ will not.  This contradicts that $S$ is the peak set.  (2) A vertex in $\overline{K}_n$ (not in $S$) will be labeled by an element in $\mathcal{I}$. In this case, that vertex would be a peak contradicting that $S$ is the peak set. 
We conclude that the vertices in $S$ must be labeled by the elements of  $\mathcal{I}$.

There are $|S|!$ ways to assign these labels to vertices in $S$, and in any of these labelings the label $m-|S|$ must be assigned in $V(G)$, otherwise there will be an additional peak in $V(\ol{K_n})$. There are $|V(G)|$ such labelings, each of which guarantees that none of the vertices in $V(\ol{K_n})\backslash S$ is a peak. None of the remaining vertices are peaks, so we are free to assign them the labels $1,2, \ldots, m-|S|-2,  m-|S|-1$ in any order. This completes the proof. 
\end{proof}

A similar result is acquired when replacing $\ol{K_n}$ by $K_n$ in Proposition \ref{theorem:nullgraphjoin} and, as the proof is analogous, we omit the details.
 
\begin{proposition} \label{prop:complete}
Let $G$ be an arbitrary graph and let $G' = K_n \vee G$. If $S\subseteq V(K_n)$, then 
\begin{align*}
   |P(S;G')| = \begin{cases}
    \left(|V(G')|-1\right)! & \text{if } |S|=1\\
    0 &\text{otherwise.}
    \end{cases}
\end{align*}
\end{proposition}

Many graph families can be constructed 
as the join of two graphs including: 
star graphs, cone graphs, fan graphs, complete bipartite,  windmill graphs, 
and wheel graphs.
Propositions \ref{theorem:nullgraphjoin} and 
\ref{prop:complete} can be applied to give closed formulas describing $|P(S;G)|$ 
for any admissible peak set $S \subseteq V(G)$ 
when $G$ is a star graph, ternary star graph, complete bipartite graph, or a Dutch windmill graph, 
and also give some closed formulas 
describing $|P(S;G)|$ when $G$ is a wheel graph, fan graph, and cone graph. We compile these results in
Table \ref{tab:results} below.

\begin{table}
\centering
\begin{adjustbox}{max width=.8\textwidth}
\begin{tabular}{|c|c|c|}
\hline
 {\bf{Graph}}&{\bf{Example}}&{\bf{Results}}\\\hline\hline
\multirow{3}{*}{$\begin{matrix}\text{Star graph:}\\  \mathcal{S}_n=K_1\vee \overline{K}_{n-1} \end{matrix}$}&
\begin{tikzpicture}[scale=0.75,baseline=20]

   \tikzstyle{hollow node}=[draw=white]
   \tikzset{blue node/.style={circle,fill=cyan,draw=black, inner sep=2.5}
}
   \node(18)[blue node] at (4,0){{\tiny{$v_1$}}};
   \node(19)[blue node] at (4.62,0.78){};
   \node(20)[blue node] at (3.78,0.97){};
   \node(21)[blue node] at (3.10,0.43){};
   \node(22)[blue node] at (3.10,-0.43){};
   \node(23)[blue node] at (3.78,-0.97){};
   \node(24)[blue node] at (4.62,-0.78){};
   \node(25)[blue node] at (5,0){};
\node at (4.15,-1.5) {$\mathcal{S}_8$};   
   \draw(18)--(19);
   \draw(18)--(20);
   \draw(18)--(21);
   \draw(18)--(22);
   \draw(18)--(23);
   \draw(18)--(24);
   \draw(18)--(25); 
   
\end{tikzpicture}
&
\multirow{3}{*}{$\left|P(S;\mathcal{S}_n)\right|=\begin{cases}
 (n-1)!(n-1) &\mbox{ if } S=\emptyset\\
 (n-1)! &\mbox{ if } S=\{v_1\}\\0&\mbox{otherwise.}\end{cases}$}\\
 \hline
  \multirow{3}{*}{$\begin{matrix}\text{Ternary Stars graph:}\\  k\mathcal{S}_n=K_k\vee\overline{K}_{n-k} \end{matrix}$}
  &
 \begin{tikzpicture}[scale=0.75,baseline=40pt]
   \tikzstyle{hollow node}=[draw=white]
   \tikzset{blue node/.style={circle,draw=black, fill=cyan, inner sep=2.5}}
\node(a)[blue node] at (-4,0){};
   \node(b)[blue node] at (-3,0){};
   \node(c)[blue node] at (-2,0){};
   \node(d)[blue node] at (-3.5,1){};
   \node(e)[blue node] at (-2.5,1){};
   \node(f)[blue node] at (-3,2){};
   \node at (-3,-.75) {$3\mathcal{S}_6$};   
   \draw(a)--(b);
   \draw(a)--(e);
   \draw(a)--(d);
   \draw(b)--(c);
   \draw(b)--(f);
   \draw(c)--(e);
   \draw(c)--(d);
   \draw(d)--(f);
   \draw(e)--(f);
   \draw(d)--(e);
   \draw(b)--(e);
   \draw(b)--(d);
   \end{tikzpicture}&\multirow{3}{*}{$\left|P(S; k\mathcal{S}_n )\right|=\begin{cases} 
 (N-1)! & \text{ if } S \subset V(K_k) \text{ and } |S|=1 \\
 |S|! \cdot |k| \cdot (N-|S|-1)!  & \text{ if } S \subset V(\ol{K_{N-k}}) \\ 
 0 & \text{ otherwise. }\\
\end{cases}$}
\\
\hline
\multirow{3}{*}{$\begin{matrix}\text{Complete Bipartite:}\\  K_{n,m}=\overline{K}_n \vee \overline{K}_m \end{matrix}$}
  &
  \begin{tikzpicture}[scale=0.75,baseline=50]

   \tikzstyle{hollow node}=[draw=white]
   \tikzset{blue node/.style={circle,fill=cyan,draw=black, inner sep=2.5}
}
   \node(1)[blue node] at (4.5,2.5){};
   \node(2)[blue node] at (4.5,1.5){};
   \node(3)[blue node] at (4.5,.5){};
   \node(4)[blue node] at (4.5,-.5){};
   \node(11)[blue node] at (3,0){};
   \node(12)[blue node] at (3,1){};
   \node(13)[blue node] at (3,2){};
   \node at (3.75,-.75) {$K_{3,4}$};   
   \draw(1)--(11);
   \draw(1)--(12);
   \draw(1)--(13);
   \draw(2)--(11);
   \draw(2)--(12);
   \draw(2)--(13);
   \draw(3)--(11);
   \draw(3)--(12);
   \draw(3)--(13);
   \draw(4)--(11);
   \draw(4)--(12);
   \draw(4)--(13);
\end{tikzpicture}
  &
\multirow{3}{*}{$|P(S;K_{n,m})| = \begin{cases} 
|S|!\cdot m \cdot (m+n-|S|-1)! & \text{ if } S \subset V(\overline{K_n}) \\ 
|S|!\cdot n \cdot (m+n-|S|-1)! & \text{ if } S \subset V(\overline{K_m}) \\                
 0 & \text{otherwise.}
               \end{cases}$}\\
               \hline
\multirow{3}{*}{$\begin{matrix}\text{Dutch Windmill graph:}\\  M^n_k=\left(\cup_{i=1}^nP_2\right)\vee\overline{K}_1 \end{matrix}$}
&
\begin{tikzpicture}[scale=0.75,baseline=20]

   \tikzstyle{hollow node}=[draw=white]
   \tikzset{blue node/.style={circle,fill=cyan,draw=black, inner sep=2.5}
}
   \node(11)[blue node] at (0,0){};
   \node(12)[blue node] at (0.5,0.866){};
   \node(13)[blue node] at (-0.5,0.866){};
   \node(14)[blue node] at (-1,0){};
   \node(15)[blue node] at (-0.5,-0.866){};
   \node(16)[blue node] at (0.5,-0.866){};
   \node(17)[blue node] at (1,0){};
   \node at (0,-1.5) {$M^3_3$};   
   \draw(11)--(12);
   \draw(11)--(13);
   \draw(11)--(14);
   \draw(11)--(15);
   \draw(11)--(16);
   \draw(11)--(17);
   \draw(12)--(13);
   \draw(14)--(15);
   \draw(16)--(17);
   \end{tikzpicture}
&
\multirow{3}{*}{$\left|P(\{v_i\};M_n^k)\right|=
(n-1)!$ where $v_i$ is a any noncentral vertex in $M_n^k$.}\\
               \hline
\multirow{3}{*}{$\begin{matrix}\text{Wheel graph:}\\  W_n=C_n\vee\overline{K}_1 \end{matrix}$}
&
\begin{tikzpicture}[scale=0.75,baseline=18]

   \tikzstyle{hollow node}=[draw=white]
   \tikzset{blue node/.style={circle,fill=cyan,draw=black, inner sep=2.5}
}
   \node(18)[blue node] at (4,0){};
   \node(19)[blue node] at (4.62,0.78){};
   \node(20)[blue node] at (3.78,0.97){};
   \node(21)[blue node] at (3.10,0.43){};
   \node(22)[blue node] at (3.10,-0.43){};
   \node(23)[blue node] at (3.78,-0.97){};
   \node(24)[blue node] at (4.62,-0.78){};
   \node(25)[blue node] at (5,0){};
\node at (4.15,-1.5) {$W_8$};   
   \draw(18)--(19);
   \draw(18)--(20);
   \draw(18)--(21);
   \draw(18)--(22);
   \draw(18)--(23);
   \draw(18)--(24);
   \draw(18)--(25); 
   \draw(19)--(20);
   \draw(20)--(21);
   \draw(21)--(22);
   \draw(22)--(23);
   \draw(23)--(24);
   \draw(24)--(25);
   \draw(25)--(19);
\end{tikzpicture}
&
\multirow{2}{*}{$\left|P(S;W_n)\right|=
n!$ if $S=\{v_1\}$ is the central vertex in $W_n$.}
\\
               \hline
\multirow{3}{*}{$\begin{matrix}\text{Fan graph:}\\  F_{n,m}=P_n \vee \overline{K}_m  \end{matrix}$}

&
\begin{tikzpicture}[scale=0.75,baseline=50]

   \tikzstyle{hollow node}=[draw=white]
   \tikzset{blue node/.style={circle,fill=cyan,draw=black, inner sep=2.5}
}
   \node(1)[blue node] at (4.5,2.5){};
   \node(2)[blue node] at (4.5,1.5){};
   \node(3)[blue node] at (4.5,.5){};
   \node(4)[blue node] at (4.5,-.5){};
   \node(11)[blue node] at (3,0){};
   \node(12)[blue node] at (3,1){};
   \node(13)[blue node] at (3,2){};
   \node at (3.75,-.75) {$F_{3,4}$};   
   \draw(1)--(11);
   \draw(1)--(12);
   \draw(1)--(13);
   \draw(2)--(11);
   \draw(2)--(12);
   \draw(2)--(13);
   \draw(3)--(11);
   \draw(3)--(12);
   \draw(3)--(13);
   \draw(4)--(11);
   \draw(4)--(12);
   \draw(4)--(13);
   \draw(11)--(12);
   \draw(12)--(13);
\end{tikzpicture}
&
\multirow{2}{*}{
$|P(S;F_{n,m})|=|S|!\cdot n\cdot (n+m-|S|-1)!$ if $S\subset V(\overline{K}_m)$.}
\\
               \hline
\multirow{3}{*}{$\begin{matrix}\text{$m$-gonal}\\\text{$n$-cone graph:}\\  C_{m,n}=C_n\vee\ol{K}_m  \end{matrix}$}

&
\begin{tikzpicture}[scale=0.5,baseline=35]

   \tikzstyle{hollow node}=[draw=white]
   \tikzset{blue node/.style={circle,fill=cyan,draw=black, inner sep=2.5}
}
   \node(18)[blue node] at (3,2.5){};
   \node(1)[blue node] at (5,2.5){};
   \node(19)[blue node] at (4.62,0.78){};
   \node(20)[blue node] at (3.78,0.97){};
   \node(21)[blue node] at (3.10,0.43){};
   \node(22)[blue node] at (3.10,-0.43){};
   \node(23)[blue node] at (3.78,-0.97){};
   \node(24)[blue node] at (4.62,-0.78){};
   \node(25)[blue node] at (5,0){};
\node at (4.15,-2) {$C_{7,2}$};   
   \draw(18)--(19);
   \draw(18)--(20);
   \draw(18)--(21);
   \draw(18)--(22);
   \draw(18)--(23);
   \draw(18)--(24);
   \draw(18)--(25); 
   \draw(19)--(20);
   \draw(20)--(21);
   \draw(21)--(22);
   \draw(22)--(23);
   \draw(23)--(24);
   \draw(24)--(25);
   \draw(25)--(19);
   \draw(1)--(19);
   \draw(1)--(20);
   \draw(1)--(21);
   \draw(1)--(22);
   \draw(1)--(23);
   \draw(1)--(24);
   \draw(1)--(25); 
   
\end{tikzpicture}
&
\multirow{2}{*}{
$|P(S;C_{m,n})|=|S|!\cdot n\cdot (n+m-|S|-1)! \text{ if } S\subset V(\overline{K}_m)$.}\\
              \hline
\end{tabular}
\end{adjustbox}
\caption{Corollaries to Propositions \ref{theorem:nullgraphjoin} and \ref{prop:complete}.}\label{tab:results}
\end{table}

\section{Future Considerations}\label{sec:fc}
One point for further investigation is to develop an alternate recursive formula for computing peaks on graphs that is more amenable to computation. The current implementation of the recursion in Algorithm \ref{alg} is prohibitively slow, as the authors observed in practice when working with dense graphs with more than $10$ vertices.  This computational lag was seen in the implementation of the recursive formula in \cite{DHIO16} for peaks on path graphs.  However, the original recursive formula for these graphs developed by Billey, Burdzy and Sagan \cite{BBS13} is much more computationally efficient.  We suggest their paper as a potential source for insight on developing an efficient recursive formula for peaks on general graphs.

A central focus of this paper is investigating how peaks on graphs behave with the join operation.  Many graphs can be constructed in a similar fashion through operations on preexisting graphs.  Examples of such operations are deletion, contraction and Cartesian, Corona, rooted and tensor products.  A natural problem that arises then is to determine how enumerating peaks on graphs after applying certain graph operations is inherited from peaks on the constituent graphs themselves.

\section*{Acknowledgements}
We thank Sara Billey for introducing us to this problem.    The first author thanks the AMS-Simons Travel Grant.
The third author was partially supported by NSF grant DMS--1620202.
The second, fourth, and fifth authors gratefully acknowledges funding support from the Seidler Student/Faculty Undergraduate Scholarly Collaboration Fellowship Program at Florida Gulf Coast University. 
The sixth author thanks the Harvey Mudd College Faculty Research, Scholarship, and Creative Works Award.
We thank SACNAS for providing a supportive atmosphere for minority mathematicians through their annual national conference, which helped us cultivate this collaboration. We thank the National Security Agency (H98230-15-1-0091) and National Science Foundation (DMS-1545136) for SACNAS mini-collaboration grants that supported travel expenses.


\begin{thebibliography}{10}

\bibitem{ABN04}
M.~Aguiar, N.~Bergeron, and K.~Nyman.
\newblock The peak algebra and the descent algebras of types {B} and {D}.
\newblock {\em Trans. Amer. Math. Soc.}, 356(7):2781--2824, 2004.

\bibitem{ANO06}
M.~Aguiar, K.~Nyman, and R.~Orellana.
\newblock New results on the peak algebra.
\newblock {\em J. Algebraic Combin.}, 23(2):149--188, 2006.

\bibitem{BH06}
N.~Bergeron and C.~Hohlweg.
\newblock Colored peak algebras and {H}opf algebras.
\newblock {\em J. Algebraic Combin.}, 24(3):299--330, 2006.

\bibitem{BBPS14}
S.~Billey, K.~Burdzy, S.~Pal, and B.~Sagan.
\newblock On meteors, earthworms and {WIMP}s.
\newblock {\em Ann. Appl. Probab.}, 25(4):1729--1779, 2015.

\bibitem{BBS13}
S.~Billey, K.~Burdzy, and B.~Sagan.
\newblock Permutations with given peak set.
\newblock {\em J. of Integer Seq.}, 16, 2013.

\bibitem{BFT16}
Sara Billey, Matthew Fahrbach, and Alan Talmage.
\newblock Coefficients and {R}oots of {P}eak {P}olynomials.
\newblock {\em Exp. Math.}, 25(2):165--175, 2016.

\bibitem{CV14}
F.~Castro-Velez, A.~Diaz-Lopez, R.~Orellana, J.~Pastrana, and R.~Zevallos.
\newblock Number of permutations with same peak set for signed permutations.
\newblock {\em arXiv:1308.6621}, 2014.

\bibitem{DHIO16}
A.~Diaz-Lopez, P.~Harris, E.~Insko, and M.~Omar.
\newblock A proof of the peak polynomial positivity conjecture.
\newblock {\em arXiv:1605.01708}, 2016.

\bibitem{DHIP15}
A.~Diaz-Lopez, P.~Harris, E.~Insko, and D.~Perez-Lavin.
\newblock Peaks sets of classical coxeter groups.
\newblock {\em arXiv:1505.04479}, 2015.

\bibitem{K14}
A.~Kasraoui.
\newblock The most frequent peak set in a random permutation.
\newblock {\em arXiv:1210.5869}, 2012.

\bibitem{N03}
K.~Nyman.
\newblock The peak algebra of the symmetric group.
\newblock {\em J. Algebraic Combin.}, 17:309--322, 2003.

\bibitem{P07}
T.~K. Petersen.
\newblock Enriched {P}-partitions and peak algebras.
\newblock {\em Adv. Math.}, 209(2):561--610, 2007.

\bibitem{S78}
V.~Strehl.
\newblock Enumeration of alternating permutations according to peak sets.
\newblock {\em J. Combin. Theory Ser. A}, 24:238--240, 1978.

\end{thebibliography}

\bibliographystyle{plain} 

\label{numpages}
\end{document}